\documentclass[11pt,letterpaper,reqno,abstract=true]{amsart}

\usepackage{amsmath}
\usepackage{amsfonts}
\usepackage{amssymb}
\usepackage{amsthm}
\usepackage{amscd}

\usepackage{tikz}
\usepackage{tikz-cd}
\usetikzlibrary{cd}

\usepackage[colorlinks=true]{hyperref}
\hypersetup{urlcolor=red, citecolor=blue}

\usepackage{latexsym}
\usepackage{mathrsfs}
\usepackage{psfrag}
\usepackage[dvips]{epsfig}
\usepackage{epsfig}
\usepackage[all]{xy}         
\usepackage{wrapfig}

\usepackage[margin=2.5cm]{geometry}

\swapnumbers 

\theoremstyle{plain}                              
\newtheorem*{frobthm}{Frobenius Theorem}
\newtheorem{thm}{Theorem}

\newtheorem{defn}[thm]{Definition}
\newtheorem{lem}[thm]{Lemma}
\newtheorem{cor}[thm]{Corollary}

\theoremstyle{definition}                         

\newtheorem*{remark}{Remark} 

\theoremstyle{remark}                             

\newcommand{\R}{\mathbb{R}}                     

\newcommand{\veps}{\varepsilon}        

\newcommand{\del}{\partial}

\providecommand{\norm}[1]{\left\lVert#1\right\rVert}       

\providecommand{\abs}[1]{\left\lvert#1\right\rvert}        

\begin{document}

\title[A Frobenius integrability theorem]{A Frobenius integrability theorem for plane fields generated by
quasiconformal deformations}

\author[S. N. Simi\'c]{Slobodan N. Simi\'c}

\address{Department of Mathematics and
  Statistics \\San Jos\'e State University \\ San Jose, CA 95192-0103}

\email{slobodan.simic@sjsu.edu}

\thanks{This work was partially supported by an SJSU
Research, Scholarship, and Creative Activity grant.}

\maketitle


\begin{abstract}
  We generalize the classical Frobenius integrability theorem to plane
  fields of class $C^Q$, a regularity class introduced by Reimann
  \cite{reimann+76} for vector fields in Euclidean spaces. A $C^Q$
  vector field is uniquely integrable and its flow is a quasiconformal
  deformation. We show that an a.e. involutive $C^Q$ plane field
  (defined in a suitable way) in $\R^n$ is integrable, with integral
  manifolds of class $C^1$.
\end{abstract}







\section{Introduction}
\label{sec:introduction}

Frobenius's integrability theorem is a fundamental result in
differential topology and gives a necessary and sufficient condition
for a smooth plane field (i.e., a distribution) to be tangent to a
foliation. The result has been generalized to Lipschitz plane fields
in \cite{sns+96} and \cite{rampazzo+07}. The goal of this paper is to
further extend Frobenius's theorem to a regularity class weaker than
Lipschitz.

In \cite{reimann+76}, Reimann introduced this new regularity class for
vector fields and showed that it is situated between Lipschitz and
Zygmund. We will call the vector fields in this class Q-vector fields,
or of class $C^Q$ (see the definition below). Reimann proved that a
Q-vector field is uniquely integrable and that each time $t$ map of
its flow is quasiconformal. That is, the flow is a quasiconformal
deformation.

\begin{defn}
  A continuous vector field $ f: \R^n \rightarrow \R^{n} $ is called a \textsf{Q-vector field} or of \textsf{class} $C^Q$ if
  \begin{displaymath}
    \norm{f}_{Q} = \sup \left|
      \frac{\langle a, f(x+a)-f(x) \rangle}{|a|^{2}} -
      \frac{\langle b, f(x+b)-f(x) \rangle}{|b|^{2}}    \right| < \infty,
  \end{displaymath}
  where the supremum is taken over all $x \in \R^n$ and $|a|=|b| \neq 0$.
\end{defn}

Let 
$$
\norm{f}_Z = \sup_{x, y \in \R^n, y \neq 0} \frac{|f(x+y) + f(x-y) -2f(x)|} {|y|}
$$
and
$$
\norm{f}_L = \sup_{x \neq y} \frac{|f(x) - f(y)|}{|x-y|}
$$
denote the Zygmund and Lipschitz seminorms of $f: \R^{n} \rightarrow
\R^{n}$.

The following properties were shown in \cite{reimann+76}:

\begin{enumerate}
  
\item[(1)] $\norm{f}_{Z} \leq 4 \norm{f}_{Q} \leq 8 \norm{f}_{L}$.
  Thus every Lipschitz vector field is of class $C^Q$. In dimension one, we have $\norm{f}_{Z} = \norm{f}_{Q}.$

\item[(2)] If $f$ is a Q-vector field and $n \geq 2$, then $f$ is Frech\'{e}t 
differentiable a.e., its classical partial derivatives coincide with
its weak derivatives and are locally integrable.

\item[(3)] If $f : \R^n \to \R^n$ is a Q-vector field and $n \geq 2$, then the differential equation $\dot{x} = f(x)$ is uniquely integrable, and for every $t$, the time $t$ map of its flow is $e^{c \abs{t}}$-quasiconformal, for some $c > 0$.
\end{enumerate}

Recall also:

\begin{defn}   \label{def:qc}
  A homeomorphism $h:U \rightarrow V$ between open sets in $\R^{n}$ is
  said to be $K$-\textsf{quasiconformal} ($K$-qc) if the following
  conditions are satisfied:
  \begin{enumerate}
  \item[(a)] $h$ is absolutely continuous on almost every line segment
    in $U$ parallel to the coordinate axes;
  \item[(b)] $h$ is differentiable a.e.;
  \item[(c)] $\frac{1}{K} \norm{Dh(x)}^{n} \leq |\det Dh(x)| \leq
           K \: m(Dh(x))^{n}$, for a.e. $x \in U$,
    \end{enumerate}
         
where, for a linear map $A$, $m(A)$ denotes the ''minimum norm'' of $A$:
\begin{displaymath}
  m(A) = \min \{ |Av|: |v|=1 \}. 
\end{displaymath}
\end{defn}

The following characterization of Q-vector fields was proved in
\cite{reimann+76} (cf., Theorem 3):

\begin{thm}
  A continuous vector field $f : \R^n \to \R^n$ $(n \geq 2)$ is of
  class $C^Q$ if and only if:

  \begin{itemize}
  \item[(a)] $f$ has distributional derivatives which are locally
    integrable.
  \item[(b)] $\abs{f(x)} = O (\abs{x} \log \abs{x})$, as $\abs{x} \to
    \infty$.
    \item[(c)] The anticonformal part $Sf$ of the derivative of $f$ is
      essentially bounded, where
      \begin{displaymath}
        Sf = \frac{1}{2}[ Df + (Df)^T] - \frac{1}{n} \text{\normalfont \rmfamily Trace}(Df) I.
      \end{displaymath}
  \end{itemize}
\end{thm}

Let $X$ and $Y$ be vector fields of class $C^Q$. Since both are
a.e. differentiable, we can define their Lie bracket in the usual way:
\begin{displaymath}
  [X,Y] = DX(Y) - DY(X).
\end{displaymath}
Alternatively,
\begin{displaymath}
  [X,Y] u = X(Y u) - Y(X u),
\end{displaymath}
for every $C^\infty$ function $u : \R^n \to \R$.

For a discussion of properties of the Lie bracket in the setting of
rough vector fields (i.e., Lipschitz and below), see
\cite{colombo+tione+2021}.

\subsection*{Plane fields of class $C^Q$}
\label{sec:q-plane-fields}

We come to the question of how to define a plane field of class
$C^Q$. One possibility is to use the usual route and say that $E$ is
class $C^Q$ if it is locally spanned by $C^Q$ vector fields. This
leads to some technical difficulties so instead we opt for a slightly
stronger definition.

Let $E$ be a continuous $k$-dimensional plane field on $\R^n$ and let
$p \in \R^n$ be arbitrary. Let $H_p$ be an $(n-k)$-dimensional
coordinate plane in $\R^n$ such that $E_p$ is transverse to
$\{p \} \times H_p \subset T_p \R^n$. Continuity of $E$ implies the
existence of a neighborhood $U$ of $p$ such that for every $q \in U$,
$E_q$ remains transverse to $\{ q \} \times H_p \subset T_q \R^n$. Let
$K_p$ be the $k$-dimensional coordinate subspace of $\R^n$
complementary to $H_p$ and denote by $\pi_p$ the orthogonal projection
$\pi_p : \R^n \to K_p$. Observe that for every $q \in U$, the
restriction of $D_q \pi_p$ to $E_q$ is injective.

For an arbitrary vector field $X$ on $K_p$ with compact support in $\pi_p(U)$, define a lift $\hat{X}$ of $X$ to $E$ by
\begin{displaymath}
  \hat{X}_q = (\left. D_q \pi_p \right|_{E_q})^{-1}(X_{\pi(q)}),
\end{displaymath}
for $q \in U$, and $\hat{X}_q = 0$, otherwise. This vector field is a section of $E$.

\begin{defn}
  A continuous $k$-dimensional plane field $E$ on $\R^n$ is a
  $Q$-\textsf{plane field} or of \textsf{class} $C^Q$, if for every
  $p, U, H_p$, and $X$ as above, the following holds: if $X$ is
  $C^\infty$ and has compact support, then its lift $\hat{X}$ to $E$
  (defined as above) is a Q-vector field.
\end{defn}

Note that if $E$ is of class $C^r$ ($r \geq 1$) or Lipschitz, then
lifts to $E$ of $C^\infty$ vector fields are $C^r$ or Lipschitz, respectively.

\begin{defn}
  A plane field of class $C^Q$ is said to be \textsf{involutive} if
  for every two $C^Q$ sections $X, Y$ of $E$, their Lie bracket
  $[X,Y]$ is an a.e.  section of $E$; i.e.,
  \begin{displaymath}
    [X,Y]_p \in E_p,
  \end{displaymath}
  for a.e. $p$.
\end{defn}

Our main result is:

\begin{frobthm}
  Let $E$ be a $k$-dimensional plane field of class $C^Q$ on
  $\R^n$. If $E$ is involutive, then $E$ is integrable, in the
  following sense. For every $p \in \R^n$, there exists a cubic
  neighborhood $U = (-\veps,\veps)^n$ of $\mathbf{0}$ in $\R^n$, a
  neighborhood $V$ of $p$, and an almost everywhere differentiable
  homeomorphism
  \begin{displaymath}
    \Phi : U \to V
  \end{displaymath}
  such that $\Phi$ maps slices
  $(-\veps,\veps)^k \times \{ \text{\normalfont \rmfamily const} \}$ to integral
  manifolds of $E$ in $V$. Writing
  $\Psi = \Phi^{-1} : p \mapsto x = (x_1,\ldots, x_n)$, we have that
  the integral manifolds of $E$ in $V$ are the slices
  \begin{displaymath}
    x_{k+1} = \text{\rm constant}, \ldots, x_n = \text{\rm constant}.
  \end{displaymath}
  Every integral manifold of $E$ in $V$ lies in one of these slices
  and is of class $C^1$.
\end{frobthm}

\section{Preliminaries from the DiPerna-Lions-Ambrosio theory}

We briefly recall some basic facts from the DiPerna-Lions-Ambrosio
theory (cf., \cite{ambrosio+04, diPerna+Lions+89}) of regular
Lagrangian flows, which generalize the notion of a flow for ``rough''
vector fields. These will be needed in the proof of the main result.

The basic idea of DiPerna-Lions-Ambrosio is to exploit (via the theory
of characteristics) the connection between the ODE $\dot{x} = X(x)$,
where $X$ is a vector field, and the associated transport PDE:
\begin{displaymath}
  \frac{\del u}{\del t} + X \cdot \nabla_x u = 0,
\end{displaymath}
for a function $u = u(t,x)$.

A common definition of this generalized notion of a flow is
the following.

\begin{defn}[Regular Lagrangian flows \cite{colombo+tione+2021}]
  We say that $\phi : \R \times \R^n \to \R^n$ is a \textsf{regular
    Lagrangian flow} for a vector field $X$ on $\R^n$ if:

  \begin{enumerate}
    
  \item[(a)] For a.e. (with respect to the 1-dimensional Lebesgue
    measure) $t \in \R$ and every measure zero Borel set
    $A \subset \R^n$, the set $\phi_t^{-1}(A)$ has $n$-dimensional
    Lebesgue measure zero, where $\phi_t(x) = \phi(t,x)$.

  \item[(b)] We have $\phi_0 = \text{\rm identity}$ and for a.e.
    $x \in \R^n$, $t \mapsto \phi_t(x)$ is an absolutely continuous
    integral curve of $X$, i.e.,
    \begin{displaymath}
      \frac{d}{dt} \phi_t(x) = X(\phi_t(x)).
    \end{displaymath}
  \end{enumerate}
\end{defn}

DiPerna, Ambrosio, and Lions showed that if $X \in L^\infty \cap BV$
has essentially bounded divergence, then the regular Lagrangian flow
of $X$ exists and is unique. Regular Lagrangian flows are \emph{stable} in
the following sense: if $(X_k)$ is a sequence of smooth vector fields
such that $X_k \to X$ strongly in $L^1_{\text{loc}}$, and
$(\text{div}(X_k))$ is equibounded in $L^\infty$, then the flows
$\phi_t^k$ of $X_k$ converge strongly to $\phi_t$ in
$L^1_{\text{loc}}$, for every $t \in \R$.

If $X$ is of class $C^Q$, each time $t$-map $\phi_t$ of its flow (in
the usual sense) is $K$-quasiconformal for some $K$, hence preserves
sets of Lebesgue measure zero (see \cite{koskela+09}). Thus
$\{ \phi_t \}$ is the regular Lagrangian flow of $X$.

\begin{lem} \label{lem:lip} Let $X$ is a $C^Q$ vector field with
  divergence in $L^\infty$ and denote by $\{ \phi_t \}$ its flow. Then
  $\norm{D\phi_t}_{L^\infty} < \infty$, for each $t$, i.e., $\phi_t$
  is Lipschitz.
\end{lem}

\begin{proof}
  We follow \cite{colombo+tione+2021}. Assume for a moment that $X$ is
  smooth. Then by Liouville's Theorem, 
  \begin{displaymath}
    \frac{d}{dt} \det D\phi_t(x) = \text{div}(X)(\phi_t(x)) \cdot \det D\phi_t(x),
  \end{displaymath}
  which implies, via Gronwall's inequality that
  \begin{equation}  \label{eq:liouville}
    \exp(-T \norm{\text{div}(X)}_{L^\infty}) \leq \det D\phi_t(x) \leq
    \exp(T \norm{\text{div}(X)}_{L^\infty}),
  \end{equation}
  for all $T > 0$ and $-T \leq t \leq T$. Thus $\norm{\det
    D\phi_t}_{L^\infty} < \infty$, for each $t$. 

  To make this work for a $C^Q$ vector field $X$, consider
  $\det D\phi_t$ as a density, and take the mollifications $X^\veps$
  of $X$ (see, e.g., \cite{evans+98}); let $\phi_t^\veps$ be the flow
  of $X^\veps$. By the stability of regular Lagrangian flows, we have
  \begin{displaymath}
    \det D\phi_t^\veps \stackrel{\ast}{\rightharpoonup} \det D\phi_t
  \end{displaymath}
  as $\veps \to 0$ in $L^\infty$, which again yields
  \eqref{eq:liouville}.
  Since $\phi_t$ is also $K$-quasiconformal, it follows (see part (c) in Def. \ref{def:qc}) that
\begin{displaymath}
  \norm{D \phi_t}_{L^\infty} \leq K^{1/n} \norm{\det
    D\phi_t}_{L^\infty}^{1/n} < \infty,
\end{displaymath}
for $-T \leq t \leq T$, as desired. In particular, for any essentially
bounded vector field $Y$, $\norm{D\phi_t(Y)}_{L^\infty} < \infty$, for
$-T \leq t \leq T$.
\end{proof}

The following corollary is a consequence of Theorem 1.1 in
\cite{colombo+tione+2021} and Lemma~\ref{lem:lip}.

\begin{cor}   \label{cor:commute}
  Let $X, Y$ be bounded Q-vector fields with essentially bounded divergence, and let $\Phi = \{ \phi_t \}$, $\Psi = \{ \psi_t \}$ be their flows. Then the following statements are equivalent:

  \begin{enumerate}
  \item[(a)] $\Phi$ and $\Psi$ commute as flows, i.e.,
    \begin{displaymath}
      \phi_t \circ \psi_s(x) = \psi_s \circ \phi_t (x),
    \end{displaymath}
    for a.e. $x \in \R^n$ and all $s, t \in \R$.

  \item[(b)] $[X,Y] = \mathbf{0}$, almost everywhere.
  \end{enumerate}
\end{cor}

\section{Proof of the main result}
\label{sec:proof}

The proof is a generalization of the standard proof for the smooth
case, which can be found in, say, Lee~\cite{lee+smooth+13}. We will
show that locally $E$ admits a frame consisting of commuting $C^Q$
vector fields, the composition of whose flows then defines the desired
coordinate system.

Assume that $E$ is an involutive $k$-dimensional plane field of class
$C^Q$ and let $p = (p_1,\ldots, p_n) \in \R^n$ be an arbitrary
point. Without loss we can assume that $E_p$ is transverse to the
subspace of $T_p \R^n$ spanned by
\begin{displaymath}
  \left. \frac{\del}{\del x_{k+1}} \right|_p, \ldots, \left. \frac{\del}{\del x_n} \right|_p.
\end{displaymath}
Let $\pi : \R^n \to \R^k \times \{ \mathbf{0} \}$ be the projection
$(x_1,\ldots,x_n) \mapsto (x_1,\ldots, x_k, \overbrace{0, \ldots,
  0}^{n-k})$. Then there exists a neighborhood $W$ of $p$ such that
for every $q \in W$, $D_q \pi$ is injective when restricted to
$E_q$. Let $U$ be a neighborhood of $p$ such that
$\overline{U} \subset W$.

Let $\beta : \R^k \to \R$ be a $C^\infty$ bump function such that
$0 \leq \beta \leq 1$, $\beta = 1$ on $\pi(U)$, and $\beta = 0$ on the
complement of $\pi(W)$. For $1 \leq i \leq k$, set
\begin{displaymath}
  V_i = \beta \frac{\del}{\del x_i}.
\end{displaymath}
Then $V_i$ is a $C^\infty$ vector field on $\R^n$ with compact support
in $\pi(U)$. Denote its lift to $E$ via $\pi$ by $X_i$; i.e.,
\begin{displaymath}
  X_i(q) = (D_q \pi |_{E_q})^{-1}(V_i(\pi(q)).
\end{displaymath}
Since $E$ is of class $C^Q$, $X_i$ is also $C^Q$. Moreover, on $U$, we
have
\begin{displaymath}
  \pi_\ast([X_i,X_j]) = [\pi_\ast(X_i), \pi_\ast(X_j)] = [V_i,V_j] = 
  \left[ \frac{\del}{\del x_i}, \frac{\del}{\del x_j} \right] = 0.
\end{displaymath}
By involutivity of $E$, $[X_i,X_j]$ is a section of $E$ a.e.. Since
the restriction of $D\pi$ to $E$ is injective, it follows that
$[X_i,X_j] = 0$, a.e. on $U$.

\begin{lem}
  $X_i$ has essentially bounded divergence, for $1 \leq i \leq k$.
\end{lem}

\begin{proof}
  Fix $1 \leq i \leq k$. It follows by construction of $X_i$ that on
  $U$ we have:
  \begin{displaymath}
    X_i = \frac{\del}{\del x_i} + \sum_{j=k+1}^n a_j \frac{\del}{\del x_j},
  \end{displaymath}
  for some continuous a.e. differentiable functions $a_j$ on
  $U$. Since $X_i$ is $C^Q$, $S(X_i)$ is essentially bounded. It is easy
  to check that the $(j,j)$-component of $S(X_i)$ equals
  \begin{displaymath}
    S(X_i)_{jj} = -\frac{1}{n} \sum_{\ell = k+1}^n \frac{\del a_\ell}{\del
      x_\ell} = -\frac{1}{n} \text{div}(X_i).
  \end{displaymath}
  Thus the divergence of $X_i$ is essentially bounded.
\end{proof}

Denote the flows of $X_1,\ldots, X_k$ by $\phi_t^1, \ldots, \phi_t^k$,
respectively. By Corollary~\ref{cor:commute}, they commute.

Define a map $\Phi : \R^n \to \R^n$ by
\begin{displaymath}
  \Phi(x_1,\ldots, x_n) = (\phi^1_{x_1} \circ \cdots
  \phi^k_{x_k})(p_1,\ldots, p_k,p_{k+1}+x_{k+1}, \ldots, p_n + x_n).
\end{displaymath}
Then $\Phi$ is continuous and differentiable a.e.. We claim that there
exists $\veps > 0$  such that the restriction of $\Phi$ to the cube
$C_\veps = (-\veps,\veps)^n$ is injective. Observe that were $\Phi$ of class
$C^1$, this would follow immediately from the Inverse Function
Theorem.

Let $\veps > 0$ be small enough so that the closure of $C_\veps$ is
contained in $\Phi^{-1}(U)$. We claim that $\Phi$ is injective on the
closure of $C_\veps$. Assume that
\begin{displaymath}
  \Phi(x_1,\ldots,x_n) = \Phi(y_1,\ldots,y_n),
\end{displaymath}
i.e.,
\begin{displaymath}
  (\phi^1_{x_1} \circ \cdots
  \phi^k_{x_k})(p_1,\ldots,p_k,p_{k+1}+x_{k+1}, \ldots, p_n + x_n) = (\phi^1_{y_1} \circ \cdots
  \phi^k_{y_k})(p_1,\ldots,p_k,p_{k+1}+y_{k+1}, \ldots, p_n + y_n),
\end{displaymath}
for some $(x_1,\ldots,x_n), (y_1,\ldots,y_n) \in \overline{C}_\veps$. Denote the
flow of $\del/\del x_i$ by $\psi^i_t$. Since
$\pi \circ \phi_t^i = \psi_t^i \circ \pi$, by projecting
$\Phi(x_1,\ldots,x_n)$ and $\Phi(y_1,\ldots,y_n)$ via $\pi$, we obtain
\begin{align*}
  (\psi^1_{x_1} \circ \cdots
  \psi^k_{x_k})(p_1,\ldots,p_k) = (\psi^1_{y_1} \circ \cdots
  \psi^k_{y_k})(p_1,\ldots,p_k), 
\end{align*}
which implies that $x_i = y_i$, for $1 \leq i \leq k$.

Thus
\begin{displaymath}
  (\phi^1_{x_1} \circ \cdots
  \phi^k_{x_k})(p_1,\ldots,p_k,p_{k+1}+x_{k+1}, \ldots, p_n + x_n) = (\phi^1_{x_1} \circ \cdots
  \phi^k_{x_k})(p_1,\ldots,p_k,p_{k+1}+y_{k+1}, \ldots, p_n + y_n)
\end{displaymath}
which clearly implies $x_i = y_i$, for $k+1 \leq i \leq n$. Therefore,
$\Phi$ is 1--1 on $\overline{C}_\veps$. By the continuity of $\Phi$
and compactness of $\overline{C}_\veps$, it follows that $\Phi : C_\veps
\to \Phi(C_\veps)$ is a homeomorphism.

\medskip

Let $S_c$ be a slice $(-\veps,\veps)^k \times \{ c \} \subset C_\veps$
(where $c \in \R^{n-k}$) and let
\begin{displaymath}
  \alpha(t) = (x_1(t),\ldots, x_k(t),c)
\end{displaymath}
be an arbitrary $C^1$ path in $S_c$. Then
\begin{displaymath}
  \gamma(t) = \Phi(\alpha(t)) = \phi^1_{x_1(t)} \circ \cdots \circ \phi^k_{x_k(t)} (\text{const}).
\end{displaymath}
The chain rule and commutativity of the flows $\phi^i_t$ implies that
$\gamma'(t)$ is a linear combination of $X_1,\ldots, X_k$, hence
tangent to $E$. Therefore, $\Phi(S_c)$ is an integral
manifold of $E$.

Since the tangent bundle of each integral manifold $N$ of $E$ is a
continuous plane field (namely, $E$ restricted to $N$), it follows
that $N$ is a $C^1$ manifold. This completes the proof. \qed

\begin{remark}
  The following questions would be of interest for further exploration:
  
  \begin{enumerate}
  \item[(a)] Does the main result still hold if a $C^Q$ plane field is
    defined as locally spanned by $C^Q$ vector fields with compact
    support?

  \item[(b)] What can be said about foliations tangent to integrable
    $C^Q$ plane fields?
  \end{enumerate}
\end{remark}



\bibliographystyle{amsalpha} 




\begin{thebibliography}{Amb04}

\bibitem[Amb04]{ambrosio+04}
L.~Ambrosio, \emph{Transport equation and {C}auchy problem for {BV} vector
  fields}, Invent. Math. \textbf{158} (2004), no.~2, 227--260.

\bibitem[CT21]{colombo+tione+2021}
Maria Colombo and Riccardo Tione, \emph{On the commutativity of flows of rough
  vector fields}, Journal de Math\'{e}matiques Pures et Appliqu\'{e}es
  \textbf{1} (2021), no.~159.

\bibitem[dL89]{diPerna+Lions+89}
R.~J. diPerna and P.~L. Lions, \emph{Ordinary differential equations, transport
  theory and {S}obolev spaces}, Invent. Math. \textbf{98} (1989), no.~3,
  511--547.

\bibitem[Eva98]{evans+98}
Lawrence~C. Evans, \emph{Partial {D}ifferential {E}quations}, Graduate
  {S}tudies in {M}athematics, vol.~19, AMS, 1998.

\bibitem[Kos09]{koskela+09}
P.~Koskela, \emph{Lectures on quasiconformal and quasisymetric mappings},
  Jyv\"{a}skyl\"{a} Lectures in Mathematics (2009).

\bibitem[Lee13]{lee+smooth+13}
John~M. Lee, \emph{Introduction to {S}mooth {M}anifolds}, second ed., Grad.
  Text in Math., vol. 218, Springer, New York, 2013.

\bibitem[Ram07]{rampazzo+07}
F.~Rampazzo, \emph{Frobenius-type theorems for {L}ipschitz distributions},
  Journal of Differential Equations \textbf{243} (2007), no.~2, 270--300.

\bibitem[Rei76]{reimann+76}
H.~M. Reimann, \emph{Ordinary differential equations and quasiconformal
  mappings}, Invent. Math. \textbf{33} (1976), 247--270.

\bibitem[Sim96]{sns+96}
Slobodan~N. Simi\'c, \emph{Lipschitz distributions and {A}nosov flows}, Proc.
  Amer. Math. Soc. \textbf{124} (1996), no.~6, 1869--1877.

\end{thebibliography}


\end{document}